\documentclass[12pt, a4paper]{article}

\usepackage{amsmath, amssymb, amsfonts, amsthm}
\usepackage{geometry}
\usepackage{tikz}
\usetikzlibrary{automata, arrows.meta, positioning, bending}
\usepackage{graphicx}
\usepackage{hyperref}
\usepackage{xcolor}
\usepackage{setspace}
\usepackage{comment}
\usepackage{bm}

\hypersetup{
    colorlinks=true,
    linkcolor=blue,
    filecolor=magenta,      
    urlcolor=cyan,
    citecolor=blue,
}

\newtheorem{theorem}{Theorem}[section]
\newtheorem{lemma}[theorem]{Lemma}
\newtheorem{proposition}[theorem]{Proposition}
\newtheorem{definition}[theorem]{Definition}
\newtheorem{corollary}[theorem]{Corollary}

\theoremstyle{remark}
\newtheorem{remark}[theorem]{Remark}

\title{\textbf{Analytic Combinatorics of $d$-Set Mappings and Their Applications}}
\author{
    \textbf{Toma Diaconescu-Grabari\footnote{tomadiaconescugrabar@cmail.carleton.ca}\;\; and Daniel Panario\footnote{daniel@math.carleton.ca}} \\
    \textit{School of Mathematics and Statistics, Carleton University}}
\date{}

\begin{document}

\maketitle

\begin{abstract}
A $d$-set mapping is a function acting on a domain $X$ equipped with a partition into $d$ disjoint subsets. While standard functions represent $1$-set mappings, generalizations to arbitrary $d$-partite structures appear naturally across discrete mathematics. In this paper, we develop an analytic combinatorial framework to quantify the functional graphs of these mappings. By leveraging generating functions and singularity analysis, we derive exact asymptotic expansions for macroscopic graph properties as the cardinality of $X$ tends to infinity, including the expected number of connected components, cyclic nodes, and tail lengths. We demonstrate the efficacy of this framework by recovering the classical bipartite mapping results of Hansen and Jaworski, and successfully generalize these mechanisms to arbitrary $d$-set mappings, providing the foundational architecture to establish their probabilistic limit laws.
\end{abstract}

\section{Introduction}
\label{sec:introduction}
Random mappings $f: V \rightarrow V$ on a finite set $V$, along with their associated functional graphs, serve as one of the canonical models linking discrete probability and analytic combinatorics. Historically, the structural properties of random permutations and mappings have been extensively studied, notably in the classical probabilistic works of Goncharov \cite{goncharov1962} and Kolchin \cite{kolchin1986}. Bringing this closer to modern analytic combinatorics, Flajolet and Odlyzko \cite{flajolet1990random} established a comprehensive framework to extract the exact asymptotic properties of standard uniform random mappings. In standard uniform models, mappings transition across $V$ entirely unconstrained. However, localized rules and structural constraints frequently arise across discrete mathematics. Notable examples include the partitioned random walks utilized in Pollard's Rho algorithm for the Discrete Logarithm Problem \cite{teske1998}, the block-shuffling mechanisms of Generalized Feistel Networks \cite{luby1988}, and the auditing of Pseudo-Random Number Generators (PRNGs) to ensure their underlying functional graphs emulate uniform random baselines \cite{flajolet1990random}. A generic study of the functional graphs of generalized cyclotomic mappings over finite fields is provided in \cite{bors2026functional}. In this paper, we define a $d$-set mapping as a function acting on a domain $V$ partitioned into $d$ disjoint subsets, and we develop an analytic combinatorial framework to extract the singular expansions governing their macroscopic component structures, proving that their normalized component sizes converge in distribution to a family of Poisson-Dirichlet limit laws.

The bipartite case ($d=2$) serves as a natural baseline for these constrained systems. Hansen and Jaworski \cite{hansen2000large} established limit laws for bipartite random mappings, proving their component sizes converge to a Poisson-Dirichlet distribution. While their work resolves two-dimensional constraints, a generalized analytic framework capable of extracting limit laws for arbitrary $d$-set mappings has, to the best of our knowledge, remained unexplored.

To bridge this gap, we leverage a dual methodology. First, we utilize the symbolic method of analytic combinatorics to translate the combinatorial definitions of $d$-set mappings into systems of exponential generating functions. By applying the Drmota-Lalley-Woods (DLW) theorem to the interdependent tree equations, we resolve the singular behavior of the cyclic components. Second, we invoke the probabilistic framework for logarithmic combinatorial assemblies formalized by Arratia, Barbour, and Tavaré \cite{arratia2003logarithmic} to translate this singular behavior into limit laws for the component size distributions.

Applying this framework, we first recover the bipartite mapping results (Corollary \ref{cor:bipartite_macro}), providing an alternative analytic-probabilistic proof of the limit laws established by Hansen and Jaworski \cite{hansen2000large}. We then extend these algebraic mechanisms to $d$ dimensions. For irreducible transition matrices where mappings remain globally connected, we prove that the cycle generating functions exhibit a universal logarithmic singularity with parameter $\theta = 1/2$ (Theorem \ref{thm:thetaonehalf}), yielding macroscopic component sizes governed by the Poisson-Dirichlet distribution $\mathcal{PD}(1/2)$.   

Finally, by analyzing reducible mappings that fragment into distinct communicating classes, we isolate the exact phase transitions that occur when the graph decomposes. We establish that when a mapping contains $k$ dominant terminal communicating classes sharing the minimal dominant singularity, the logarithmic parameter shifts to $\theta = k/2$. Consequently, the component sizes transition from the universal baseline into the generalized $\mathcal{PD}(k/2)$ family of limit laws (Theorem \ref{thm:thetakhalf}).

The remainder of this paper is organized as follows: Section 2 establishes the analytic combinatorics framework, the probabilistic conditioning mechanics for logarithmic assemblies, and the defining properties of the exp-log class. Section 3 analyzes the bipartite case, serving as an accessible two-dimensional example to explicitly demonstrate the algebraic mechanics and establish baseline limit laws. Section 4 generalizes this architecture to arbitrary irreducible $d$-set mappings, isolating their universal logarithmic singularity. Section 5 evaluates reducible mappings, classifying the exact phase transitions that occur when the functional graph fragments into distinct communicating classes. Finally, Section 6 provides brief conclusions and notes potential cryptographic applications to be explored in future work.

\section{Background and Methodology}
This section establishes the well-known tools required to analyze the component structure of random mappings. We rely on a dual approach: we first utilize the symbolic method of analytic combinatorics to  define the generating functions of our mapping structures, and we then invoke the probabilistic framework of combinatorial assemblies to extract limit laws for the component size distributions, from both macroscopic and microscopic perspectives.

\subsection{Analytic Combinatorics Framework}
\label{sec:analytic_framework}

We adopt the standard analytic combinatorics framework pioneered by Flajolet and Sedgewick \cite{flajolet2009analytic}. The framework relies on translating symbolic definitions of combinatorial classes into exponential generating functions (EGFs) of the form $A(z) = \sum_{n \ge 0} A_n \frac{z^n}{n!}$, and subsequently extracting asymptotic properties via transfer theorems \cite{flajolet1990singularity}.

A mapping $f: V \to V$ on a finite set $V$ corresponds to a functional graph---a finite directed graph where every vertex has out-degree one. Such a graph decomposes into an unordered set of connected components, where each component is a directed cycle of rooted trees. Let $T(z)$, $C(z)$, and $M(z)$ denote the EGFs for trees, cyclic components, and mappings, respectively. The symbolic constructions $\mathcal{C} = \text{CYC}(\mathcal{T})$ and $\mathcal{M} = \text{SET}(\mathcal{C})$ yield the direct functional relationships:
\begin{equation*}
    C(z) = \ln\left(\frac{1}{1 - T(z)}\right) \quad \text{and} \quad M(z) = \exp(C(z)).
\end{equation*}

This direct exponential relationship between the cyclic components $C(z)$ and the macroscopic mapping $M(z)$ exemplifies a universal combinatorial structure formally modeled by the exp-log class \cite{flajolet1990gaussian}. Originally defined by Flajolet and Soria, the exp-log class characterizes assemblies where the irreducible building blocks exhibit a logarithmic singularity and are combined via the exponential SET operator. This structural framework extends well beyond random mappings, governing diverse mathematical objects such as permutations decomposed into distinct cycles and polynomials over finite fields decomposed into irreducible factors. Further asymptotic properties of the exp-log class, including the precise size distributions of the largest and smallest components, have been extensively analyzed in subsequent literature, notably in the thesis of Gourdon \cite{gourdon1996combinatoire} and the collaborative works of Panario and Richmond \cite{panario2001smallest}.

As established in Flajolet and Odlyzko \cite{flajolet1990random}, the tree function for uniform random mappings satisfies the symbolic construction $\mathcal{T} = \mathcal{Z} \star \text{SET}(\mathcal{T})$, yielding the Cayley tree equation $T(z) = z \exp(T(z))$. However, to analyze $d$-set partitioned mappings, we must evaluate a system of interdependent functional equations defining a vector of tree generating functions $\mathbf{T}(z) = (T_1(z), \dots, T_d(z))^T$. 

Once the local singular expansions of such systems are resolved, we translate them into asymptotic coefficient estimates using the standard transfer theorems.

\begin{theorem}[\cite{flajolet1990singularity, flajolet2009analytic}]
\label{thm:transfer}
Let $f(z)$ be an analytic function in a $\Delta$-domain defined as $\{z \colon |z| < R, z \neq \rho, |\arg(z - \rho)| > \phi\}$ for some $R > \rho$ and angle $\phi \in (0, \pi/2)$. If $f(z)$ admits an algebraic singularity of the form $f(z) \sim K\left(1 - \frac{z}{\rho}\right)^{-\alpha}$ for $\alpha \notin \{0, -1, -2, \dots\}$, its coefficients satisfy
\begin{equation*}
    [z^n]f(z) \sim K \frac{n^{\alpha-1}}{\Gamma(\alpha)} \rho^{-n},
\end{equation*}
where $\Gamma$ is the gamma function. Furthermore, for a logarithmic singularity of the form $f(z) \sim \kappa \ln\left(\frac{1}{1 - z/\rho}\right)$, the coefficients transfer directly to
\begin{equation*}
    [z^n]f(z) \sim \kappa \frac{\rho^{-n}}{n}.
\end{equation*}
\end{theorem}

\subsection{Probabilistic Limit Theorems for Assemblies}
\label{sec:probabilistic_framework}

To derive the joint distribution of component sizes, we adopt the probabilistic framework for logarithmic combinatorial structures formalized by Arratia, Barbour, and Tavar\'e \cite{arratia2003logarithmic}. 

In this framework, a combinatorial assembly is a structure composed of an unordered collection of connected components. This probabilistic concept corresponds exactly to the $\text{SET}$ construction in analytic combinatorics. For random mappings, these components are the directed cycles of rooted trees. If we let $C_k^{(n)}$ denote the discrete random variable representing the number of components of size $k$ in a mapping of total size $n$, these counts are bound by the partition constraint:
\begin{equation*}
    \sum_{k=1}^n k C_k^{(n)} = n.
\end{equation*}

The exact summation of the sizes to $n$ forces the variables $C_k^{(n)}$ to be dependent. To circumvent this dependency, we bypass the discrete graph entirely and model the component counts as independent random variables, subsequently conditioning on the total size. 

Let $(Z_1, Z_2, \dots, Z_n)$ be a sequence of mutually independent Poisson random variables where $Z_k \sim \operatorname{Po}(\lambda_k)$. The expected value is defined via the analytic component generating function $C(z)$ and the dominant singularity $\rho$:
\begin{equation}
    \lambda_k = [z^k]C(z) \rho^k. \label{eq:poisson_expected}
\end{equation}

The conditioning relation of the framework in \cite{arratia2003logarithmic} guarantees that the joint distribution of the dependent component counts $(C_1^{(n)}, \dots, C_n^{(n)})$ is identical to that of the independent Poisson variables, conditioned on the event that their size-weighted sum equals $n$:
\begin{equation}
    \mathcal{L}\left(C_1^{(n)}, \dots, C_n^{(n)}\right) = \mathcal{L}\left(Z_1, \dots, Z_n \ \colon \sum_{k=1}^n k Z_k = n\right), \label{eq:conditioning_relation}
\end{equation}
where $\mathcal{L}(\cdot)$ denotes the probability law (or joint distribution) of the random vector.

\begin{remark}
    The emergence of the Poisson distribution in this relation is not arbitrary, but rather a direct consequence of the $\text{SET}$ operator. As an assembly constitutes an unordered collection, constructing a configuration with exactly $c$ components of the same size necessitates dividing the combinatorial weight by $c!$ to prevent overcounting permutations. The Poisson probability mass function inherently mirrors this $1/c!$ exponential symmetry. This structural correspondence naturally extends to other decomposable classes: combinatorial structures allowing unrestricted component repetition (\emph{multisets}) are governed by negative binomial distributions, whereas structures that forbid component repetition (\emph{selections}) are governed by binomial distributions.
\end{remark}

To extract limit laws from this conditioning relation, the combinatorial structure must exhibit a specific asymptotic behavior in its expected component counts, known as the \emph{logarithmic condition}.

\begin{definition}
\label{def:logarithmic_condition}
An assembly is defined as logarithmic with parameter $\theta > 0$ if the expected counts of its components satisfy:
\begin{equation}
    k \lambda_k \to \theta \quad \text{as} \quad k \to \infty.
\end{equation}
\end{definition}

We observe the direct correspondence between Theorem \ref{thm:transfer} (transfer theorem) and Definition \ref{def:logarithmic_condition} (logarithmic condition). If the cycle generating function $C(z)$ exhibits a logarithmic singularity of the form $C(z) \sim \theta \ln\left(\frac{1 - z}{\rho}\right)^{-1}$, the transfer theorem dictates that $[z^k]C(z) \rho^k \sim \frac{\theta}{k}$. Thus, the singularity of the generating function forces the assembly to satisfy the logarithmic condition. 

\subsection{Multivariate Analytic Systems and the Jacobian}
\label{sec:multivariate_systems}

To resolve the dominant singularities of interdependent functional equations, we rely on the Multivariate Analytic Implicit Function Theorem and its extensions for non-negative power series.

\begin{theorem}[{\cite[Theorem B.6]{flajolet2009analytic}}]
\label{thm:maift}
Let $f_i(x_1, \dots, x_m ; z_1, \dots, z_p)$, with $1 \leq i \leq m$, be analytic functions in the neighbourhood of a point $x_j = a_j$, $1 \leq j \leq m$, and $z_k = c_k$, $1 \leq k \leq p$. Assume that the Jacobian determinant defined as
\begin{equation*}
    \Delta := \det \left( \frac{\partial f_i}{\partial x_j} \right)
\end{equation*}
is non-zero at the point considered. Then the equations (in the $x_j$, $1 \leq j \leq m$)
\begin{equation*}
    y_i = f_i(x_1, \dots, x_m ; z_1, \dots, z_p), \quad 1 \leq i \leq m,
\end{equation*}
admit a solution with the $x_j$ near to the $a_j$, $1 \leq j \leq m$, when the $z_k$ are sufficiently near to the $c_k$, $1 \leq k \leq p$, and the $y_i$ near to the $b_i := f_i(a_1, \dots, a_m ; c_1, \dots, c_p)$, $1 \leq i \leq m$: one has
\begin{equation*}
    x_j = g_j(y_1, \dots, y_m ; z_1, \dots, z_p), \quad 1 \leq j \leq m,
\end{equation*}
where each $g_j$, $1 \leq j \leq m$, is analytic in a neighbourhood of the point $(b_1, \dots, b_m ; c_1, \dots, c_p)$.
\end{theorem}

Consider a system of functional equations written in the implicit form $\mathbf{F}(z, \mathbf{y}) = \mathbf{y} - \mathbf{\Phi}(z, \mathbf{y}) = \mathbf{0}$. By Theorem \ref{thm:maift}, the solution vector $\mathbf{y}(z)$ remains analytic at a point $z$ provided that the Jacobian matrix of the implicit system, $\frac{\partial \mathbf{F}}{\partial \mathbf{y}} = \mathbf{I} - \mathbf{J}$ (where $\mathbf{J} = \left[ \frac{\partial \Phi_i}{\partial y_j} \right]$), is invertible. 

Consequently, any breakdown of analyticity must coincide with this matrix becoming singular. For systems of combinatorial generating functions, the coefficients are inherently non-negative. By Pringsheim's theorem \cite[Theorem IV.1]{flajolet2009analytic}, the dominant singularity $\rho$ of such a system must lie strictly on the positive real axis. Evaluating the system for real $z > 0$ guarantees that the Jacobian $\mathbf{J}$ is a non-negative matrix. Thus, by the Perron-Frobenius theorem \cite[Note V.34]{flajolet2009analytic}, its largest eigenvalue (the spectral radius) is a real, positive number. Because the underlying power series grow monotonically along the positive real axis, the spectral radius of $\mathbf{J}$ strictly increases as $z$ grows. Bell, Burris, and Yeats \cite{bell2011characteristic} formalized that this breakdown of analyticity occurs exactly when this monotonically growing spectral radius reaches $1$. This establishes $1$ as an eigenvalue of $\mathbf{J}$, forcing the operator $\mathbf{I} - \mathbf{J}$ to lose invertibility and yielding the boundary equation $\det(\mathbf{I} - \mathbf{J}) = 0$.

\section{The Bipartite Case}
\label{sec:bipartite}

Before analyzing the general $d$-set partition, we first establish the bipartite mapping model. This two-dimensional example serves as an introduction to the mechanics governing the multidimensional generalization.

\subsection{Symbolic Construction and the Transition Matrix}

The bipartite mapping property defined below models the fundamental behavior of elements transitioning between distinct subsets. As introduced in Section \ref{sec:introduction}, such structural constraints arise naturally across discrete mathematics; for instance, the two-dimensional bipartite case explicitly models the functional graphs of generalized cyclotomic mappings over finite fields, where algebraic operations map elements strictly between distinct cyclotomic cosets \cite{bors2026functional}. By initially isolating this bipartite property, we can clearly establish the analytic architecture and macroscopic component statistics for a large class of mappings. While this section focuses purely on the bipartite baseline, this foundational framework paves the way for the non-bipartite and multi-dimensional configurations introduced earlier.

\begin{definition}
Let $f:V\to V$ be a mapping on a finite set $V$, and let $(S_0,S_1)$ be a bipartition of $V$. The function $f$ satisfies the \emph{bipartite mapping property} if $f(S_0)\subseteq S_1$ and $f(S_1)\subseteq S_0$.
\end{definition}

\begin{figure}[htbp]
    \centering
    \begin{tikzpicture}[
        >={Stealth[scale=1.2]}, 
        node distance=3cm,     
        on grid,
        every state/.style={thick, fill=blue!5, draw=blue!80!black, minimum size=1cm}
    ]

    \begin{scope}[shift={(6,0)}]
        \node[state] (C0) {$S_0$};
        \node[state] (C1) [right=of C0] {$S_1$};
        \path[->, thick] 
        (C0) edge [bend left=20] (C1)
        (C1) edge [bend left=20] (C0);
    \end{scope}

    \end{tikzpicture}
    \caption{Macroscopic transition diagram illustrating the bipartite mapping property.}
    \label{fig:bipartite_mapping}
\end{figure}
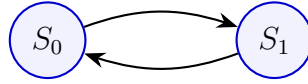

Building upon the standard symbolic construction of functional graphs established in Section \ref{sec:analytic_framework}, we adapt the tree structures to reflect the bipartition of their root nodes. Let $\mathcal{Z}_0$ and $\mathcal{Z}_1$ denote the atomic classes of nodes belonging to $S_0$ and $S_1$, respectively. The interdependent tree classes are defined symmetrically:
\begin{align}
    \mathcal{T}_0 &= \mathcal{Z}_0 \star \text{SET}(\mathcal{T}_1), \\
    \mathcal{T}_1 &= \mathcal{Z}_1 \star \text{SET}(\mathcal{T}_0).
\end{align}

Let $w_{ij}$ represent the adjacency indicator from $S_i$ to $S_j$, where $i,j\in \{0,1\}$. By definition of the bipartition, nodes map to the opposite partition, forcing the bipartite transition matrix $W$ to be anti-diagonal:
\begin{equation}
    W = \begin{pmatrix} w_{00} & w_{01} \\ w_{10} & w_{11} \end{pmatrix} = \begin{pmatrix} 0 & 1 \\ 1 & 0 \end{pmatrix}.
\end{equation}

Let $p \in (0,1)$ represent the proportion of nodes in $S_0$, and $1-p$ the proportion of nodes in $S_1$. Tracking the weight of nodes from $S_0$ with $pz$ and nodes from $S_1$ with $(1-p)z$, we translate the symbolic tree equations into a system of EGFs:
\begin{subequations}
\label{eq:treeeqs}
\begin{align}
    T_0(z) &= pz \exp(T_1(z)), \label{eq:T0_bipartite} \\
    T_1(z) &= (1-p)z \exp(T_0(z)). \label{eq:T1_bipartite}
\end{align}
\end{subequations}

Before evaluating the cyclic components, we observe how this bivariate system collapses at its symmetric and boundary limits.

\begin{remark}[Symmetric and Boundary Cases]
If $p = 1/2$, symmetry forces $T_0(z) = T_1(z) = T(z)$, and the system collapses to $T(z) = \frac{1}{2}z \exp(T(z))$, the standard Cayley tree scaled by $1/2$. 

If $p \to 1$, the partition probability forces all nodes into $S_0$. The tree equations degenerate to $T_1(z) = 0$ and $T_0(z) = z$, meaning a tree in $S_0$ is merely an isolated root node. Conversely, if $p \to 0$, the exact symmetric collapse occurs in the opposite partition, forcing all nodes into $S_1$ and yielding $T_0(z) = 0$ and $T_1(z) = z$.
\end{remark}

\subsection{The Transfer Matrix Method and Generating Functions}

To derive the explicit algebraic form of the cycle generating function $C(z)$, we view the cyclic components as directed necklaces of rooted trees. By encoding the bipartite transitions into a transfer matrix \cite[Section V.6]{flajolet2009analytic}, we enumerate these closed alternating walks.

\begin{proposition}
\label{prop:bipartite_gfs}
Let $A(z) = \operatorname{diag}(T_0(z), T_1(z))W$ be the transfer matrix encoding the weight of trees planted along a directed bipartite cycle. The exponential generating functions for the cyclic components $C(z)$ and the full bipartite mapping $M(z)$ are:
\begin{align}
    C(z) &= \ln\left(\frac{1}{1-T_0(z)T_1(z)}\right), \label{eq:bipartite_C} \\
    M(z) &= \frac{1}{1-T_0(z)T_1(z)}. \label{eq:bipartite_M}
\end{align}
\end{proposition}

\begin{proof}
We construct the transfer matrix $A(z)$ by multiplying the diagonal matrix of tree generating functions by the transition matrix $W$:
\begin{equation*}
    A(z) = \begin{pmatrix} T_0(z) & 0 \\ 0 & T_1(z) \end{pmatrix} \begin{pmatrix} 0 & 1 \\ 1 & 0 \end{pmatrix} = \begin{pmatrix} 0 & T_0(z) \\ T_1(z) & 0 \end{pmatrix}.
\end{equation*}

The $(i,j)$ entry of $A(z)^k$ enumerates the valid walks of length $k$ from partition $i$ to partition $j$, assigning the planted tree weight of the origin partition prior to each transition step. The trace $\operatorname{Tr}(A(z)^k)$ sums the diagonal entries, isolating the closed walks of length $k$. 

A cyclic component (a necklace) of length $k$ is formed by gluing the ends of a closed walk. As there are $k$ indistinguishable starting points under the cyclic shift, the total weight of length $k$ cycles is $\frac{1}{k} \operatorname{Tr}(A(z)^k)$. Summing over all possible lengths yields the full cycle generating function:
\begin{align*}
    C(z) &= \sum_{k=1}^\infty \frac{\operatorname{Tr}(A(z)^k)}{k} = \operatorname{Tr}\left( \sum_{k=1}^\infty \frac{A(z)^k}{k} \right)= \operatorname{Tr}\Big( -\ln(I-A(z)) \Big).
\end{align*}

Applying Jacobi's formula, $\operatorname{Tr}(\ln(X)) = \ln(\det(X))$, we transform the matrix trace into a scalar determinant:
\begin{equation*}
    C(z) = \ln\left(\frac{1}{\det(I-A(z))}\right).
\end{equation*}

Evaluating the determinant for the bipartite case yields $\det(I-A(z)) = 1 - T_0(z)T_1(z)$, satisfying Equation \eqref{eq:bipartite_C}. 

Finally, since the complete functional graph is an unordered collection of these cyclic components, we apply the symbolic $\text{SET}$ operator, which corresponds to the exponential function:
\begin{equation*}
    M(z) = \exp(C(z)) = \frac{1}{1 - T_0(z) T_1(z)},
\end{equation*}
satisfying Equation \eqref{eq:bipartite_M} and concluding the proof.
\end{proof}

With the generating functions $C(z)$ and $M(z)$ established, the asymptotic behavior of the bipartite mapping is now entirely dictated by the dominant singularity of the interdependent tree system $(T_0(z), T_1(z))$. To resolve this singularity, we evaluate the system's analytic properties.

\subsection{Analytic Properties and the DLW Theorem}

To extract the asymptotic behavior of $M(z)$, we must determine the dominant singularity of the tree system. For an arbitrary $p \neq 1/2$, the asymmetry yields $T_0(z) \neq T_1(z)$, meaning the system must be resolved using the Drmota-Lalley-Woods (DLW) theorem.

\begin{definition}[{\cite[p. 489, and p. 493 Note VII.29]{flajolet2009analytic}}]
\label{def:analytic_system}
Consider a nonlinear system of functional equations $\mathbf{y}(z) = \mathbf{\Phi}(z, \mathbf{y}(z))$, where $\mathbf{\Phi}$ is a vector of functions analytic at the origin. The system is defined by the following algebraic properties:
\begin{itemize}
    \item \textbf{Algebraic positivity (a-positive):} The analytic functions $\Phi_j$ have non-negative MacLaurin coefficients.
    \item \textbf{Algebraic properness (a-proper):} The system satisfies a strict Lipschitz condition $d(\mathbf{\Phi}(\mathbf{y}), \mathbf{\Phi}(\mathbf{s})) < K d(\mathbf{y}, \mathbf{s})$ for some $K < 1$, where $d$ is the formal distance metric $d(\mathbf{u}, \mathbf{v}) = 2^{-\text{val}(\mathbf{u}-\mathbf{v})}$, and the valuation is the exponent of the lowest non-zero monomial.
    \item \textbf{Algebraic irreducibility (a-irreducible):} The dependency graph of the system is strongly connected.
    \item \textbf{Algebraic aperiodicity (a-aperiodic):} For each component solution $y_j(z)$, the support of its MacLaurin coefficient sequence $[z^n]y_j(z)$ contains all sufficiently large integers, meaning the period is exactly $1$.
\end{itemize}
\end{definition}

\begin{remark}
The properties of Definition \ref{def:analytic_system} were formulated by Flajolet and Sedgewick \cite[p. 489]{flajolet2009analytic} strictly for polynomial systems, guaranteeing a unique sequence solution via the fixed-point theorem on the complete metric space of formal power series. However, the universal square-root singularities guaranteed by these properties extend to systems of non-negative analytic functions, provided a positive solution exists within their domain of analyticity. This extension, originally pioneered by Drmota \cite{drmota1997systems}, was formally generalized to multidimensional non-negative power series systems by Bell, Burris, and Yeats \cite{bell2011characteristic}. This analytic extension is what permits the inclusion of the exponential operator $\exp(T_i)$ in our functional graph derivations.
\end{remark}

\begin{theorem}[{\cite{drmota1997systems}, \cite[Theorem VII.6]{flajolet2009analytic}}]
\label{thm:dlw}
Let $\mathbf{y}(z) = \mathbf{\Phi}(z, \mathbf{y}(z))$ be a nonlinear analytic system that is a-positive, a-proper, and a-irreducible. Then all component solutions $y_i(z)$ share the same radius of convergence $\rho < \infty$. 

Furthermore, if the system is a-aperiodic, $\rho$ is the unique dominant singularity on the circle of convergence, and in a neighborhood of $\rho$, each $y_i(z)$ admits a universal square-root expansion:
\begin{equation*}
    y_i(z) = g_i(z) - h_i(z)\sqrt{1 - \frac{z}{\rho}},
\end{equation*}
where $g_i$ and $h_i$ are analytic at $\rho$, and $h_i(\rho) > 0$.
\end{theorem}

The system of equations of interest, $\mathbf{T}(z) = \mathbf{\Phi}(z, \mathbf{T}(z))$, is explicitly given by:$$\mathbf{T}(z) = \begin{pmatrix} T_0(z) \\ T_1(z) \end{pmatrix} = \begin{pmatrix} pz \exp(T_1(z)) \\ (1-p)z \exp(T_0(z)) \end{pmatrix}$$corresponding to Equations \eqref{eq:treeeqs}.

\begin{lemma}
\label{lem:bipartite_dlw_conditions}
The system $\mathbf{\Phi}$ as defined above is a-positive, a-proper, a-irreducible, and a-aperiodic for all $p \in (0,1)$.
\end{lemma}

\begin{proof}
We verify the four conditions of Definition \ref{def:analytic_system}:
\begin{itemize}
    \item \textbf{a-positive:} The Taylor expansion of the exponential function has strictly positive MacLaurin coefficients. Since $p \in (0,1)$, the partition probabilities $p$ and $1-p$ are strictly positive, guaranteeing that all components of the system have non-negative coefficients.
    
    \item \textbf{a-proper:} Let $\mathbf{T} = (T_0, T_1)$ and $\mathbf{S} = (S_0, S_1)$ be vectors in $\mathbb{C}[[z]]^2$. Let $V = \text{val}(\mathbf{T} - \mathbf{S})$ be their minimum valuation, implying $T_i - S_i = \mathcal{O}(z^V)$. Applying the transformation $\Phi_0(\mathbf{T}) = pz\exp(T_1)$ yields:
    \begin{equation*}
        \Phi_0(\mathbf{T}) - \Phi_0(\mathbf{S}) = pz \big(\exp(T_1) - \exp(S_1)\big).
    \end{equation*}
    Because the lowest order nonzero term of $(\exp(T_1) - \exp(S_1))$ is governed by $T_1 - S_1$, its valuation is at least $V$. Multiplying by the explicit $z$ factor shifts every term up by one degree, strictly increasing the valuation:
    \begin{equation*}
        \text{val}\big(\Phi_0(\mathbf{T}) - \Phi_0(\mathbf{S})\big) \ge V + 1.
    \end{equation*}
    By symmetry, this applies to $\Phi_1$ as well. Translating this back into the formal distance metric:
    \begin{equation*}
        d\big(\mathbf{\Phi}(\mathbf{T}), \mathbf{\Phi}(\mathbf{S})\big) \le 2^{-(V+1)} = \frac{1}{2} d(\mathbf{T}, \mathbf{S}).
    \end{equation*}
    The system satisfies the strict Lipschitz condition with $K = 1/2 < 1$, proving it is a contraction mapping and thus a-proper.
    
    \item \textbf{a-irreducible:} The functional dependency graph consists of the directed edges $T_0 \to T_1$ and $T_1 \to T_0$. This two-node graph is strongly connected.
    
    \item \textbf{a-aperiodic:} A valid bipartite tree can be constructed for any total integer size $n \ge 1$ (for example, a root node in $S_0$ connected to $n-1$ child nodes in $S_1$). Therefore, the support of the coefficient sequence $[z^n]T_i(z)$ contains all integers $n \ge 1$, meaning the period is exactly $1$.
\end{itemize}
\end{proof}

Having satisfied all four analytic conditions, the system falls under Theorem \ref{thm:dlw}, which guarantees that $T_0(z)$ and $T_1(z)$ share a unique dominant singularity $\rho$, with a universal square-root expansion:
\begin{equation}
    \begin{aligned}
        T_0(z) &\sim \tau_0 - c_0 \sqrt{1 - z/\rho}, \\
        T_1(z) &\sim \tau_1 - c_1 \sqrt{1 - z/\rho}.
    \end{aligned}
    \label{eq:DLWexpansions}
\end{equation}
where $\tau_i = T_i(\rho)$. As established in Section \ref{sec:multivariate_systems}, this unique tree singularity occurs when the spectral radius of the system's Jacobian matrix $\mathbf{J} = \left[ \frac{\partial \Phi_i}{\partial T_j} \right]$ is equal to $1$, or when $\det(\mathbf{I} - \mathbf{J}) = 0$. For our bipartite system, computing the partial derivatives yields:
\begin{equation*}
    \mathbf{J} = \begin{bmatrix} \frac{\partial \Phi_0}{\partial T_0} & \frac{\partial \Phi_0}{\partial T_1} \\ \frac{\partial \Phi_1}{\partial T_0} & \frac{\partial \Phi_1}{\partial T_1} \end{bmatrix} = \begin{bmatrix} 0 & pz\exp(T_1) \\ (1-p)z\exp(T_0) & 0 \end{bmatrix} = \begin{bmatrix} 0 & T_0(z) \\ T_1(z) & 0 \end{bmatrix}.
\end{equation*}
The determinant of $\mathbf{I} - \mathbf{J}$ is $1 - T_0(z)T_1(z)$. The dominant singularity $\rho$ occurs exactly when this determinant vanishes, forcing the boundary condition $1 - T_0(\rho)T_1(\rho) = 0$, yielding $\tau_0 \tau_1 = 1$. 

We observe that this tree singularity condition $1 - T_0 T_1 = 0$ is identical to the denominator in the EGFs $C(z)$ and $M(z)$. We prove in Section \ref{sec:d_set_generalization} that this spectral equivalence between the cycles and the trees is a universal property for any arbitrary dimension $d$.

\begin{theorem}
\label{thm:bipartite_log}
The cycle generating function $C(z)$ exhibits a universal logarithmic singularity with parameter $\theta = 1/2$, independent of the arbitrary partition proportion $p \in (0,1)$.
\end{theorem}
\begin{proof}
To evaluate $C(z)$ near the dominant singularity $\rho$, we expand the denominator $1 - T_0(z)T_1(z)$ using equation (\ref{eq:DLWexpansions}):
\begin{align*}
    T_0(z)T_1(z) &\sim \left(\tau_0 - c_0\sqrt{1 - z/\rho}\right)\left(\tau_1 - c_1\sqrt{1 - z/\rho}\right) \\
    &\sim \tau_0 \tau_1 - (\tau_0 c_1 + \tau_1 c_0)\sqrt{1 - z/\rho} + \mathcal{O}\left(1 - z/\rho\right).
\end{align*}
By the a-positivity of the system, the singular evaluations $\tau_i = T_i(\rho)$ and the expansion coefficients $c_i$ are positive real numbers. Therefore, by defining the positive constant $K = \tau_0 c_1 + \tau_1 c_0 > 0$, the denominator reduces to a square-root singularity:
\begin{equation}
    1 - T_0(z)T_1(z) \sim K \sqrt{1 - \frac{z}{\rho}}.
\end{equation}
Substituting this asymptotic equivalence back into the cycle function defined in Equation \eqref{eq:bipartite_C}, we isolate the dominant singular term:
\begin{equation}
    C(z) = \ln\left(\frac{1}{K \sqrt{1 - z/\rho}}\right) = -\ln(K) + \frac{1}{2} \ln\left(\frac{1}{1 - z/\rho}\right) \sim \frac{1}{2} \ln\left(\frac{1}{1 - z/\rho}\right).
\end{equation}
Although the exact values of $\tau_i$, $c_i$, and $K$ depend on the partition proportion $p$, they are isolated into the constant $-\ln(K)$. Thus, the logarithmic parameter is $\theta = 1/2$ for all $p \in (0,1)$.
\end{proof}

\subsection{Probabilistic Limit Laws for Bipartite Mappings}
\label{sec:bipartite_limit}

With the logarithmic singularity of the cycle generating function established in Theorem \ref{thm:bipartite_log}, we directly apply the probabilistic framework defined in Section \ref{sec:probabilistic_framework}. Satisfying the logarithmic condition guarantees the convergence of both the small component counts and the macroscopic component sizes, governed by the standard limit theorems for logarithmic assemblies.

\begin{theorem}[{\cite[Theorem 3.2]{arratia2003logarithmic}}]
\label{thm:general_micro}
Let an assembly satisfy the logarithmic condition with parameter $\theta > 0$. Let $b = b(n)$ be a sequence of integers such that $b \to \infty$ and $b = o(n)$ as $n \to \infty$. The total variation distance between the joint distribution of the dependent component counts and the independent Poisson process converges to zero:
\begin{equation}
    d_{TV}\Big(\mathcal{L}(C_1^{(n)}, \dots, C_b^{(n)}), \mathcal{L}(Z_1, \dots, Z_b)\Big) \to 0
\end{equation}
where $Z_k \sim \operatorname{Po}(\lambda_k)$.
\end{theorem}

\begin{theorem}[{\cite[Theorems 6.8 and 6.12]{arratia2003logarithmic}}]
\label{thm:general_macro}
Let $L_1^{(n)} \ge L_2^{(n)} \ge \cdots$ denote the sizes of the largest connected components of an assembly of size $n$, ordered decreasingly. If the assembly satisfies the logarithmic condition with parameter $\theta > 0$, the joint distribution of the normalized component sizes converges in distribution to the Poisson-Dirichlet distribution $\mathcal{PD}(\theta)$:
\begin{equation}
    n^{-1}\left( L_1^{(n)}, L_2^{(n)}, \dots \right) \xrightarrow{d} \mathcal{PD}(\theta) \quad \text{as} \quad n \to \infty.
\end{equation}
\end{theorem}

To invoke these limit theorems for bipartite random mappings, we must prove the mapping satisfies the logarithmic condition.

\begin{proposition}
\label{pro:bipartite_log_cond}
The component counting process of a bipartite random mapping with arbitrary partition proportion $p \in (0,1)$ satisfies the logarithmic condition with parameter $\theta = 1/2$.
\end{proposition}
\begin{proof}
From Equation \eqref{eq:poisson_expected}, the expected value of the unconditioned Poisson component counts is given by $\lambda_k = [z^k]C(z)\rho^k$. By Theorem \ref{thm:bipartite_log}, the cycle function admits the singular expansion $C(z) \sim \frac{1}{2} \ln\left(\frac{1}{1 - z/\rho}\right)$. Applying Theorem \ref{thm:transfer} yields:
\begin{equation*}
    \lambda_k \sim \frac{1}{2k}.
\end{equation*}
Multiplying by $k$ and taking the limit as $k \to \infty$ evaluates the asymptotic intensity:
\begin{equation*}
    \lim_{k \to \infty} k \lambda_k = \frac{1}{2}.
\end{equation*}
This confirms the assembly satisfies Definition \ref{def:logarithmic_condition} with $\theta = 1/2$.
\end{proof}

Because the bipartite mapping is a logarithmic assembly, the probabilistic limit laws follow directly from Proposition \ref{pro:bipartite_log_cond} and the fundamental theorems for assemblies (Theorems \ref{thm:general_micro} and \ref{thm:general_macro}). We give them next as corollaries.

\begin{corollary}
\label{cor:bipartite_micro}
For a bipartite random mapping, the initial segment of components decouples, converging to the independent Poisson process where $\lambda_k \sim \frac{1}{2k}$.
\end{corollary}

\begin{corollary}
\label{cor:bipartite_macro}
For a bipartite random mapping with arbitrary $p \in (0,1)$, the joint distribution of the normalized macroscopic component sizes converges to the Poisson-Dirichlet distribution:
\begin{equation}
    n^{-1}\left( L_1^{(n)}, L_2^{(n)}, \dots \right) \xrightarrow{d} \mathcal{PD}(1/2) .
\end{equation}
\end{corollary}

\begin{remark}
    Corollary \ref{cor:bipartite_macro} provides an alternative analytic-probabilistic proof of the macroscopic limit originally established by Hansen and Jaworski \cite{hansen2000large} for arbitrary partition distributions.
\end{remark}

While the limits established in this section fully resolve the strictly bipartite case, they inherently rely on the transition matrix operating as a two-dimensional anti-diagonal operator. However, many of the applied functional graphs introduced in Section \ref{sec:introduction} fall into non-bipartite cases involving a broader number of partitions or varying transition constraints. Having established the fundamental analytic and probabilistic mechanics here, we now generalize this framework. The limit laws for non-bipartite, arbitrary $d$-set mappings with globally connected (irreducible) transitions are handled in Section \ref{sec:d_set_generalization}, while the more complex reducible mappings that fragment into distinct communicating classes are fully addressed in Section \ref{sec:reducible_matrices}. 

\section{Generalization to Arbitrary $d$-Set Mappings}
\label{sec:d_set_generalization}

Having established the mechanics for the bipartite case, we now generalize the framework to functional graphs partitioned into an arbitrary number of sets. 

Let $V$ be a finite set of $n$ elements, partitioned into $d$ disjoint subsets $S_1, S_2, \dots, S_d$, and let $p_i \in (0,1)$ denote the proportion of nodes in $S_i$. We observe that $\sum_{i=1}^d p_i = 1$. 

We define a $d$-partite random mapping $f: V \to V$ with transition probabilities defined by an irreducible $d \times d$ stochastic matrix $\mathbf{W}$. The entries $w_{ij}$ represent the probability that an element in partition $S_i$ maps to an element in partition $S_j$:
\begin{equation}
    w_{ij} = \frac{|\{v \in S_i \mid f(v) \in S_j\}|}{|S_i|} = \frac{|S_i \cap f^{-1}(S_j)|}{|S_i|}.
\end{equation}

\begin{remark}
We restrict our initial analysis to irreducible transition matrices $\mathbf{W}$ to guarantee that the associated functional dependency graph is strongly connected. This strict topological condition is required to satisfy the \textit{a-irreducible} property of the Drmota-Lalley-Woods theorem. The mechanics of relaxing this condition for reducible matrices—by decomposing the graph and analyzing the dominant communicating classes—is addressed in Section \ref{sec:reducible_matrices}.
\end{remark}

\subsection{DLW Theorem in $d$-Dimensions}

The symbolic construction of the component trees expands naturally into a system of $d$ interdependent equations. Let $\mathbf{T}(z) = (T_1(z), T_2(z), \dots, T_d(z))^T$ be the vector of exponential generating functions for the trees rooted in each respective partition. The system is defined by the transformation $\mathbf{T}(z) = \mathbf{\Phi}(z, \mathbf{T}(z))$, where the $i$-th component is:
\begin{equation}
    T_i(z) = p_i z \exp\left( \sum_{j=1}^d w_{ij} T_j(z) \right).
\end{equation}

Defining the diagonal partition matrix $\mathbf{P} = \operatorname{diag}(p_1, p_2, \dots, p_d)$, the system of equations can be expressed in vector form using the element-wise exponential:
\begin{equation}
\label{eq:systemofdequations}
\mathbf{T}(z) = z \mathbf{P} \exp(\mathbf{W} \mathbf{T}(z)).
\end{equation}

\begin{lemma}
\label{lem:d_set_dlw}
Let $\mathbf{W}$ be an irreducible $d \times d$ stochastic matrix. The $d$-variate analytic system $\mathbf{\Phi}$ is a-positive, a-proper, a-irreducible, and a-aperiodic for any valid partition distribution $(p_1, \dots, p_d)$.
\end{lemma}
\begin{proof}
We verify the four analytic conditions of Definition \ref{def:analytic_system}:
\begin{itemize}
    \item \textbf{a-positive:} Since the partition proportions $p_i$ and the matrix weights $w_{ij}$ are non-negative, and the MacLaurin series of the exponential function has strictly positive coefficients, all components $\Phi_i$ possess non-negative coefficients.
    \item \textbf{a-proper:} Let $\mathbf{T}, \mathbf{S} \in \mathbb{C}[[z]]^d$ with minimum formal valuation $V = \text{val}(\mathbf{T} - \mathbf{S})$. Evaluating the difference of the operator yields:
    \begin{equation*}
        \Phi_i(\mathbf{T}) - \Phi_i(\mathbf{S}) = p_i z \left( \exp\Big( \sum_{j=1}^d w_{ij} T_j \Big) - \exp\Big( \sum_{j=1}^d w_{ij} S_j \Big) \right).
    \end{equation*}
    The lowest order non-zero term of the difference of the exponentials has a valuation of at least $V$. Multiplication by the independent variable $z$ shifts the degree, yielding $\text{val}(\mathbf{\Phi}(\mathbf{T}) - \mathbf{\Phi}(\mathbf{S})) \ge V + 1$. Translating to the formal distance metric, $d(\mathbf{\Phi}(\mathbf{T}), \mathbf{\Phi}(\mathbf{S})) \le \frac{1}{2} d(\mathbf{T}, \mathbf{S})$, proving the system is a contraction mapping.
    \item \textbf{a-irreducible:} This is satisfied by assumption. 
    \item \textbf{a-aperiodic:} The exponential function generates all non-negative powers of its argument. Consequently, the support of the coefficient sequence $[z^n]T_i(z)$ contains all integers $n \ge 1$, establishing a sequence period of exactly $1$.
\end{itemize}
\end{proof}

By the DLW theorem and its analytic extension, the system $\mathbf{T}(z)$ possesses a unique dominant singularity $\rho$ on its circle of convergence. Furthermore, there exist functions $h_i$ analytic at the origin such that in a neighborhood of $\rho$:
\begin{equation}
    \mathbf{T}(z) = \left(h_1\left(\sqrt{1-z/\rho}\right), \dots, h_d\left(\sqrt{1-z/\rho}\right)\right)^T,
\end{equation}
where each component admits the singular expansion $T_i(z) \sim \tau_i - c_i\sqrt{1-z/\rho}$, with $\tau_i = T_i(\rho)$.

\subsection{The Universal Logarithmic Singularity}

We now construct the $d \times d$ transfer matrix $\mathbf{A}(z) =  \operatorname{diag}(\mathbf{T}(z))\mathbf{W}$. Under the transfer matrix method, the cycle generating function $C(z)$ is given by:
\begin{equation}
    C(z) = \sum_{k=1}^\infty \frac{\operatorname{Tr}(\mathbf{A}(z)^k)}{k} = -\ln\Big(\det(\mathbf{I} - \mathbf{A}(z))\Big).
\end{equation}

To extract the asymptotic behavior of the macroscopic cycles, we must evaluate the roots of the determinant $\det(\mathbf{I} - \mathbf{A}(z))$. However, the singular limit of the underlying component trees is governed by a different matrix: the Jacobian $\mathbf{J}$ of Equation (\ref{eq:systemofdequations}). To guarantee that the cycle EGF $C(z)$ and the tree EGFs $T_i(z)$ share the same dominant singularity $\rho$, we must establish an equivalence.

\begin{lemma}
\label{lem:jacobianequivalence}
The dominant singularity $\rho$ of the analytic system is the unique positive real value $z$ that satisfies the equation $\det(\mathbf{I} - \mathbf{J}) = 0$, where $\mathbf{J} = \operatorname{diag}(\mathbf{T}(z)) \mathbf{W}$ is the Jacobian matrix of the system.
\end{lemma}
\begin{proof}
The Jacobian of the system of equations $\mathbf{T}=\mathbf{\Phi}(z,\mathbf{T})$ is the matrix $\mathbf{J}$ where $[\mathbf{J}]_{ij}=\frac{\partial \Phi_i}{\partial T_j}$. By the chain rule, we obtain
\begin{equation*}
    \frac{\partial \Phi_i}{\partial T_j} = \frac{\partial}{\partial T_j} \left[ z p_i \exp\left( \sum_{k=1}^d w_{ik} T_k \right) \right] = z p_i \exp\left( \sum_{k=1}^d w_{ik} T_k \right) w_{ij} = T_i(z) w_{ij}.
\end{equation*}
In matrix notation, multiplying each entry $w_{ij}$ by the component $T_i$ corresponds exactly to scaling the rows of the transition matrix $\mathbf{W}$ by the elements of $\mathbf{T}$. Thus the Jacobian is $\mathbf{J} = \operatorname{diag}(\mathbf{T}(z)) \mathbf{W}$. We also note that $w_{ij}$ and $T_i(z)$ are non-negative, thus $\mathbf{J}$ is also a non-negative matrix.

Now define the implicit system $\mathbf{F}(z, \mathbf{T}) = \mathbf{T}-\mathbf{\Phi}(z,\mathbf{T})=\mathbf{T} - z \mathbf{P} \exp(\mathbf{W} \mathbf{T}) = \mathbf{0}$. By Theorem \ref{thm:maift}, the solution vector $\mathbf{T}(z)$ remains analytic at a point $z$ as long as the Jacobian matrix of $\mathbf{F}$ with respect to $\mathbf{T}$ is invertible. 

The derivative of the implicit system is $\frac{\partial \mathbf{F}}{\partial \mathbf{T}} = \mathbf{I} - \mathbf{J}$. As established in Section \ref{sec:multivariate_systems}, while the implicit function theorem guarantees analyticity when this matrix is invertible, the converse holds for non-negative analytic systems \cite{bell2011characteristic}: the breakdown of analyticity at the dominant singularity $\rho$ occurs exactly when the spectral radius of the non-negative Jacobian $\mathbf{J}$ reaches $1$. This is equivalent to the linear operator $\mathbf{I} - \mathbf{J}$ losing invertibility, yielding the boundary equation $\det(\mathbf{I} - \mathbf{J}) = 0$.
\end{proof}

The structural identity between the transfer matrix $\mathbf{A}(z)$ and the Jacobian $\mathbf{J}$ guarantees they share the same determinant, establishing the equivalence:
\begin{equation}
    \det(\mathbf{I} - \mathbf{A}(z)) = \det(\mathbf{I} - \operatorname{diag}(\mathbf{T})\mathbf{W}) = \det(\mathbf{I} - \mathbf{J}).
\end{equation}
We define this as $D(\mathbf{T}) = \det(\mathbf{I} - \mathbf{J})$.

\begin{theorem}
\label{thm:thetaonehalf}
For any $d$-partite random mapping governed by an irreducible transition matrix $\mathbf{W}$, the cycle EGF $C(z)$ exhibits a universal logarithmic singularity with parameter $\theta = 1/2$.
\end{theorem}
\begin{proof}
By Lemma \ref{lem:jacobianequivalence}, as $\mathbf{T}(z)$ approaches the singular limit $\bm{\tau} = (\tau_1, \dots, \tau_d)^T$, the spectral radius of the Jacobian reaches $1$, forcing the determinant $D(\bm{\tau}) = 0$. We evaluate $D(\mathbf{T}(z))$ near the singularity using a first-order multivariate Taylor expansion:
\begin{equation*}
    D(\mathbf{T}(z)) \sim \sum_{i=1}^d \frac{\partial D}{\partial T_i}(\bm{\tau}) \big(T_i(z) - \tau_i\big).
\end{equation*}
Substituting the expansion $T_i(z) - \tau_i \sim -c_i\sqrt{1-z/\rho}$, we extract the dominant asymptotic behavior:
\begin{equation*}
    D(\mathbf{T}(z)) \sim \left( -\sum_{i=1}^d c_i \frac{\partial D}{\partial T_i}(\bm{\tau}) \right) \sqrt{1-z/\rho}.
\end{equation*}
By the monotonic growth of the combinatorial generating functions $T_i(z)$ along the positive real axis, the singular expansion coefficients $c_i$ must be positive. 

To prove that $\frac{\partial D}{\partial T_i}(\tau) < 0$, we apply Jacobi's formula, which defines the derivative of the determinant as:
\begin{equation*}
    \frac{\partial D}{\partial T_i}(\tau) = \operatorname{Tr}\left( \operatorname{adj}(\mathbf{I}-\mathbf{J}) \frac{\partial(\mathbf{I}-\mathbf{J})}{\partial T_i} \right).
\end{equation*}
By the assumed irreducibility of $\mathbf{W}$, the Jacobian $\mathbf{J}$ evaluated at the singularity $\tau$ is a non-negative, irreducible matrix with a spectral radius of 1. By the Perron-Frobenius theorem \cite[Note V.34]{flajolet2009analytic}, 1 is a simple dominant eigenvalue, making the eigenvalue 0 of $\mathbf{I}-\mathbf{J}$ simple. Consequently, the adjugate matrix $\operatorname{adj}(\mathbf{I}-\mathbf{J})$ is a rank-1 matrix constructed from the outer product of the strictly positive left and right Perron eigenvectors, guaranteeing every entry is strictly positive.

Since $\mathbf{J} = \operatorname{diag}(\mathbf{T})\mathbf{W}$, the partial derivative $\frac{\partial(\mathbf{I}-\mathbf{J})}{\partial T_i} = -\frac{\partial \mathbf{J}}{\partial T_i}$ yields a matrix containing the negative transition weights $-w_{ik}$ strictly along its $i$-th row, and zeros elsewhere. Multiplying the strictly positive adjugate matrix by this derivative matrix yields a product with strictly negative entries along its diagonal for all non-zero $w_{ik}$. Because $\mathbf{W}$ is irreducible, each row contains at least one positive weight, guaranteeing the matrix trace is strictly negative. Therefore, $\frac{\partial D}{\partial T_i}(\tau) < 0$ and, $K := -\sum_{i=1}^d c_i \frac{\partial D}{\partial T_i}(\bm{\tau})>0$.

Substituting $D(\mathbf{T}(z)) \sim K\sqrt{1-z/\rho}$ into the cycle EGF, we isolate the dominant singular term:
\begin{equation}
    C(z) = -\ln\Big(D(\mathbf{T}(z))\Big) \sim -\ln\left(K \sqrt{1 - z/\rho}\right) = -\ln(K) + \frac{1}{2} \ln\left(\frac{1}{1 - z/\rho}\right).
\end{equation}
The scalar term $-\ln(K)$ is asymptotically negligible, leaving the singularity entirely governed by $\frac{1}{2} \ln\left(\frac{1}{1 - z/\rho}\right)$. Thus, the parameter evaluates universally to $\theta = 1/2$ for all valid dimensions $d$ and irreducible configurations $\mathbf{W}$. 
\end{proof}

\section{Reducible Matrices and Communicating Classes}
\label{sec:reducible_matrices}
In Theorem \ref{thm:thetaonehalf}, the universal parameter $\theta = 1/2$ is fundamentally reliant on the transition matrix $\mathbf{W}$ being irreducible. When $\mathbf{W}$ is reducible, its associated dependency graph is no longer strongly connected, meaning the functional graph fragments into distinct terminal and transient communicating classes. We first evaluate the two reducible cases in dimension $d=2$, and then treat the general $d$ case.

\subsection{Reducible Boundaries in Two Dimensions}

\begin{lemma}
Suppose a $2$-set mapping is entirely disconnected, meaning elements only map within their own sets ($f(S_0) \subseteq S_0$ and $f(S_1) \subseteq S_1$). If the sets are perfectly symmetric ($p = 1/2$), the component sizes shift into the $\mathcal{PD}(1)$ universality class.
\end{lemma}

\begin{proof}
The transition matrix is the identity matrix, $\mathbf{W} = \begin{bmatrix} 1 & 0 \\ 0 & 1 \end{bmatrix}$. The system of tree generating functions decouples into independent Cayley structures:
\begin{align}
    T_0(z) &= pz \exp(T_0(z)), \\
    T_1(z) &= (1-p)z \exp(T_1(z)).
\end{align}
Computing the Jacobian matrix $\mathbf{J} = \operatorname{diag}(\mathbf{T}(z))\mathbf{W}$ yields a purely diagonal operator:
\begin{equation*}
    \mathbf{J} = \begin{bmatrix} T_0(z) & 0 \\ 0 & T_1(z) \end{bmatrix}.
\end{equation*}
We have $D(\mathbf{T}(z)) = \det(\mathbf{I} - \mathbf{J}) = (1 - T_0(z))(1 - T_1(z))$. 

Notice that the decoupled equations $T_0(z) = pz \exp(T_0(z))$ and $T_1(z) = (1-p)z \exp(T_1(z))$ are scaled instances of the standard Cayley tree function $T(x) = x \exp(T(x))$. Since the dominant singularity of the standard Cayley tree occurs exactly at $x = e^{-1}$, the individual trees $T_0(z)$ and $T_1(z)$ possess distinct singularities at $\rho_0 = \frac{1}{pe}$ and $\rho_1 = \frac{1}{(1-p)e}$, respectively.

If the partition is uneven ($p \neq 1/2$), the component with the larger proportion possesses the smaller singular radius ($\rho_0 < \rho_1$ when $p > 1/2$). This dominant tree reaches its singularity first, dictating the overall radius of convergence of the system. 

However, if the sets are perfectly symmetric ($p = 1/2$), their singularities coincide ($\rho_0 = \rho_1 = 2/e = \rho$). Since both trees reach their $1 - \tau_i = 0$ boundaries simultaneously, we substitute their square-root expansions $T_i(z) \sim 1 - c_i\sqrt{1-z/\rho}$ into the determinant:
\begin{equation*}
    D(\mathbf{T}(z)) \sim \left(c_0\sqrt{1 - z/\rho}\right) \left(c_1\sqrt{1 - z/\rho}\right) = c_0 c_1 \left(1 - \frac{z}{\rho}\right).
\end{equation*}
When evaluating the cycle generating function, the multiplication of the two square roots yields a purely linear singularity:
\begin{equation}
    C(z) = -\ln\Big(D(\mathbf{T}(z))\Big) \sim -\ln\left(c_0 c_1 \left(1 - \frac{z}{\rho}\right)\right) \sim \ln\left(\frac{1}{1-z/\rho}\right).
\end{equation}
The logarithmic parameter shifts from $\theta = 1/2$ to $\theta = 1$. The component sizes are no longer governed by the $\mathcal{PD}(1/2)$ distribution, shifting instead into the $\mathcal{PD}(1)$ universality class.
\end{proof}

\begin{lemma}
Suppose a 2-set mapping possesses a single absorbing state. Either $S_1$ maps into $S_0$ while $S_0$ maps to itself ($f(S_1) \subseteq S_0$ and $f(S_0) \subseteq S_0$), or symmetrically, $S_0$ maps strictly into $S_1$ while $S_1$ maps to itself ($f(S_0) \subseteq S_1$ and $f(S_1) \subseteq S_1$). In either configuration, the functional graph preserves the square-root singularity and the $\mathcal{PD}(1/2)$ limit law.
\end{lemma}

\begin{proof}
Consider the first case, where $S_0$ is the absorbing state. The transition matrix becomes $\mathbf{W} = \begin{bmatrix} 1 & 0 \\ 1 & 0 \end{bmatrix}$. The interdependent tree equations are:
\begin{align}
    T_0(z) &= pz \exp(T_0(z)), \\
    T_1(z) &= (1-p)z \exp(T_0(z)).
\end{align}
With all elements eventually terminating in $S_0$, the tree $T_1(z)$ lacks self-referential cycle formation and acts purely as transient ``dust'' feeding into the $S_0$ cores. The Jacobian matrix is calculated as:
\begin{equation*}
    \mathbf{J} = \operatorname{diag}(\mathbf{T}(z))\mathbf{W} = \begin{bmatrix} T_0(z) & 0 \\ 0 & T_1(z) \end{bmatrix} \begin{bmatrix} 1 & 0 \\ 1 & 0 \end{bmatrix} = \begin{bmatrix} T_0(z) & 0 \\ T_1(z) & 0 \end{bmatrix}.
\end{equation*}
The determinant simplifies to $D(\mathbf{T}(z)) = \det(\mathbf{I} - \mathbf{J}) = 1 - T_0(z)$. The transient tree $T_1(z)$ is completely eradicated from the spectral evaluation. Substituting the square-root expansion for $T_0(z)$ yields:
\begin{equation*}
    D(\mathbf{T}(z)) \sim c_0\sqrt{1 - \frac{z}{\rho}}.
\end{equation*}
This leaves the square-root singularity intact, preserving $\theta = 1/2$ and the $\mathcal{PD}(1/2)$ limit law. The proof for the symmetric case, where $S_1$ is the absorbing state, follows identically.
\end{proof}

\subsection{Communicating Classes in Arbitrary Dimensions}

When the $d \times d$ transition matrix $\mathbf{W}$ is reducible, its associated dependency graph is not strongly connected. Instead, the graph uniquely decomposes into a directed acyclic graph (DAG) of strongly connected components, referred to in Markov chain theory as \textit{communicating classes}. 

\begin{definition}
By permuting the indices of the partitions, any reducible transition matrix $\mathbf{W}$ can be written in block upper-triangular form, known as the Frobenius normal form:
\begin{equation}
    \mathbf{W} = \begin{bmatrix}
    \mathbf{W}_{11} & \mathbf{W}_{12} & \cdots & \mathbf{W}_{1m} \\
    \mathbf{0} & \mathbf{W}_{22} & \cdots & \mathbf{W}_{2m} \\
    \vdots & \vdots & \ddots & \vdots \\
    \mathbf{0} & \mathbf{0} & \cdots & \mathbf{W}_{mm}
    \end{bmatrix},
\end{equation}
where each diagonal block $\mathbf{W}_{kk}$ is an irreducible square matrix representing a distinct communicating class $\mathcal{C}_k$. These communicating classes fall into two categories:
\begin{itemize}
    \item \textbf{Terminal Classes (Absorbing):} A class $\mathcal{C}_k$ is terminal if it possesses no outgoing transitions to any other class. Its block row in the Frobenius form contains zeros apart from its diagonal block $\mathbf{W}_{kk}$. Elements that map into a terminal class can never leave, forming a closed, irreducible sub-system. 
    \item \textbf{Transient Classes:} A class is transient if it possesses a non-zero probability of mapping into another class. Like the $S_1$ nodes in our two-dimensional absorbing case, elements in a transient class inevitably feed into the terminal components.
\end{itemize}
For each class $\mathcal{C}_k$, let $\mathbf{T}_{\mathcal{C}_k}(z)$ denote the sub-vector of tree generating functions corresponding strictly to the partitions within $\mathcal{C}_k$.
\end{definition}

\begin{lemma}
The determinant of the block-triangular Jacobian matrix decomposes into the product of the determinants of its principal diagonal blocks.
\end{lemma}

\begin{proof}
Defined as $\mathbf{J} = \operatorname{diag}(\mathbf{T}(z))\mathbf{W}$, the Jacobian inherently preserves the block upper-triangular structure of the transition matrix:
\begin{equation*}
    \mathbf{J} = \begin{bmatrix}
    \mathbf{J}_{11} & \mathbf{J}_{12} & \cdots & \mathbf{J}_{1m} \\
    \mathbf{0} & \mathbf{J}_{22} & \cdots & \mathbf{J}_{2m} \\
    \vdots & \vdots & \ddots & \vdots \\
    \mathbf{0} & \mathbf{0} & \cdots & \mathbf{J}_{mm}
    \end{bmatrix},
\end{equation*}
where each diagonal block evaluates to $\mathbf{J}_{kk} = \operatorname{diag}(\mathbf{T}_{\mathcal{C}_k})\mathbf{W}_{kk}$. 

The determinant of a block-triangular matrix is the product of the determinants of its principal diagonal blocks; thus
\begin{equation}
    D(\mathbf{T}(z)) = \det(\mathbf{I} - \mathbf{J}) = \prod_{k=1}^m \det(\mathbf{I} - \mathbf{J}_{kk}).
\end{equation}
As each diagonal block $\mathbf{W}_{kk}$ is irreducible by definition, each corresponding sub-system $\mathbf{J}_{kk}$ obeys Theorem \ref{thm:dlw} independently.
\end{proof}

\subsection{The Phase Transition}

With the determinant factored into its independent communicating classes, we now establish the generalized boundary conditions for the phase transition of arbitrary $d$-partite mappings.

\begin{theorem}
\label{thm:thetakhalf}
Let $\mathbf{W}$ be a reducible $d \times d$ transition matrix governing a $d$-partite random mapping. If exactly $k$ isolated terminal communicating classes share the minimal dominant singularity $\rho$, the macroscopic components of the mapping satisfy the logarithmic condition with parameter $\theta = k/2$. Consequently, the normalized component sizes converge in distribution to the Poisson-Dirichlet distribution $\mathcal{PD}(k/2)$.
\end{theorem}

\begin{proof}
Let $\mathcal{C}_1, \dots, \mathcal{C}_m$ be the communicating classes of a reducible transition matrix $\mathbf{W}$. By Theorem \ref{thm:dlw}, each isolated sub-system possesses its own localized dominant singularity $\rho_j$. The global dominant singularity $\rho$ of the entire functional graph is the minimum of these localized radii:
\begin{equation*}
    \rho = \min_{1 \le j \le m} \{ \rho_j \}.
\end{equation*}

Let $\mathcal{K}$ denote the index set of the communicating classes that attain this global minimum radius $\rho$. Let $k = |\mathcal{K}|$ be the number of such classes. When we evaluate the factored characteristic determinant $D(\mathbf{T}(z))$ as $z \to \rho$, the blocks partition into two distinct asymptotic regimes:
\begin{enumerate}
    \item \textbf{Dominant Classes ($j \in \mathcal{K}$):} These classes reach their spectral radius of $1$ at exactly $z = \rho$. Their determinants evaluate to the singular expansion $\det(\mathbf{I} - \mathbf{J}_{jj}) \sim c_j \sqrt{1 - z/\rho}$.
    \item \textbf{Transient and Sub-Dominant Classes ($j \notin \mathcal{K}$):} These classes possess a strictly larger radius of convergence ($\rho_j > \rho$). At $z = \rho$, their spectral radii remain strictly less than $1$, meaning their determinants evaluate to positive, non-zero constants: $\det(\mathbf{I} - \mathbf{J}_{jj}) = K_j > 0$.
\end{enumerate}

Substituting these two regimes back into the factored determinant yields:
\begin{equation}
    D(\mathbf{T}(z)) \sim \left( \prod_{j \notin \mathcal{K}} K_j \right) \prod_{j \in \mathcal{K}} \left( c_j \sqrt{1 - \frac{z}{\rho}} \right).
\end{equation}

By grouping the constants and the singular coefficients into a single global constant $C > 0$, the $k$ independent square roots algebraically multiply:
\begin{equation}
    D(\mathbf{T}(z)) \sim C \left( \sqrt{1 - \frac{z}{\rho}} \right)^k = C \left(1 - \frac{z}{\rho}\right)^{k/2}.
\end{equation}

Extracting the cycle generating function via the logarithmic transformation $C(z) = -\ln(D(\mathbf{T}(z)))$, the exponent strictly factors out as the universal logarithmic parameter:
\begin{equation}
    C(z) \sim -\ln\left( C \left(1 - \frac{z}{\rho}\right)^{k/2} \right) \sim \frac{k}{2} \ln\left(\frac{1}{1 - z/\rho}\right).
\end{equation}
This confirms the functional graph satisfies the logarithmic condition with parameter $\theta = k/2$, concluding the proof.
\end{proof}

\section{Conclusion and Future Work}
In this paper, we established the limit laws for arbitrary $d$-set functional graphs, proving that their normalized component sizes converge in distribution to the Poisson-Dirichlet family, $\mathcal{PD}(k/2)$.

As commented in the introduction, these constrained mapping structures arise naturally in cryptographic algorithms, such as the 3-partite walks utilized in Pollard's Rho algorithm and the block-shuffling mechanisms of Generalized Feistel Networks. The analytic framework introduced in this paper may help to understand the time complexity and collision properties of these highly structured operations.

\bibliographystyle{plain} % You can change 'plain' to 'alpha' or 'abbrv' later if Panario prefers
\bibliography{references} % This points to your new references.bib file

@book{flajolet2009analytic,
  title={Analytic Combinatorics},
  author={Flajolet, Philippe and Sedgewick, Robert},
  year={2009},
  publisher={Cambridge University Press},
  address={Cambridge}
}

@book{arratia2003logarithmic,
  title={Logarithmic Combinatorial Structures: A Probabilistic Approach},
  author={Arratia, Richard and Barbour, Andrew D. and Tavar{\'e}, Simon},
  year={2003},
  publisher={European Mathematical Society},
  address={Z{\"u}rich}
}

@inproceedings{flajolet1990random,
  title={Random mapping statistics},
  author={Flajolet, Philippe and Odlyzko, Andrew M.},
  booktitle={Advances in Cryptology---EUROCRYPT '89},
  pages={329--354},
  year={1990},
  publisher={Springer}
}

@article{flajolet1990singularity,
  title={Singularity analysis of generating functions},
  author={Flajolet, Philippe and Odlyzko, Andrew M.},
  journal={SIAM Journal on Discrete Mathematics},
  volume={3},
  number={2},
  pages={216--240},
  year={1990},
  publisher={SIAM}
}

@article{hansen2000large,
  title={Large components of bipartite random mappings},
  author={Hansen, Jennie and Jaworski, Jerzy},
  journal={Random Structures \& Algorithms},
  volume={17},
  number={3-4},
  pages={317--342},
  year={2000},
  publisher={Wiley Online Library}
}

@book{bors2026functional,
  title={Functional Graphs of Generalized Cyclotomic Mappings of Finite Fields},
  author={Bors, Alexander and Panario, Daniel and Wang, Qiang},
  series={Memoirs of the European Mathematical Society},
  volume={23},
  year={2026},
  publisher={EMS Press},
  doi={10.4171/MEMS/23}
}

@article{drmota1997systems,
  title={Systems of functional equations},
  author={Drmota, Michael},
  journal={Random Structures \& Algorithms},
  volume={10},
  number={1-2},
  pages={103--124},
  year={1997},
  publisher={Wiley Online Library}
}

@article{bell2011characteristic,
  title={Characteristic points of recursive systems},
  author={Bell, Jason P. and Burris, Stanley N. and Yeats, Karen A.},
  journal={The Electronic Journal of Combinatorics},
  volume={18},
  number={1},
  pages={P54},
  year={2011}
}

@book{kolchin1986,
  title={Random Mappings},
  author={Kolchin, Valentin F.},
  year={1986},
  publisher={Optimization Software},
  address={New York}
}

@article{goncharov1962,
  title={On the field of combinatorial analysis},
  author={Goncharov, Vladimir L.},
  journal={American Mathematical Society Translations},
  series={Series 2},
  volume={19},
  pages={1--46},
  year={1962},
  publisher={American Mathematical Society}
}

@article{teske1998,
  title={Speeding up {Pollard's} rho method for computing discrete logarithms},
  author={Teske, Edlyn},
  journal={Algorithmic Number Theory},
  pages={541--554},
  year={1998},
  publisher={Springer}
}

@article{luby1988,
  title={How to construct pseudorandom permutations from pseudorandom functions},
  author={Luby, Michael and Rackoff, Charles},
  journal={SIAM Journal on Computing},
  volume={17},
  number={2},
  pages={373--386},
  year={1988},
  publisher={SIAM}
}

@article{flajolet1990gaussian,
  title={Gaussian limiting distributions for the number of components in combinatorial structures},
  author={Flajolet, Philippe and Soria, Mich{\`e}le},
  journal={Journal of Combinatorial Theory, Series A},
  volume={53},
  number={2},
  pages={165--182},
  year={1990},
  publisher={Elsevier}
}

@phdthesis{gourdon1996combinatoire,
  title={Combinatoire, algorithmique et g{\'e}om{\'e}trie des polyn{\^o}mes},
  author={Gourdon, Xavier},
  year={1996},
  school={Ecole Polytechnique}
}

@article{panario2001smallest,
  title={Smallest components in decomposable structures: exp-log class},
  author={Panario, Daniel and Richmond, Bruce},
  journal={Algorithmica},
  volume={29},
  number={1-2},
  pages={205--226},
  year={2001},
  publisher={Springer}
}

\end{document}